\documentclass[11pt]{amsart}
\usepackage{amsmath,amsfonts,amssymb,amscd,verbatim,amsthm,amsgen,xspace,graphicx,color}
\usepackage{srcltx}
\usepackage{enumerate}
\date{\today}

\newcommand{\ka}{\mathfrak{k}}
\newcommand{\p}{\mathfrak{p}}
\newcommand{\g}{\mathfrak{g}}
\newcommand{\h}{\mathfrak{h}}
\newcommand{\C}{{\ensuremath{\mathbb{C}}}}

\newcommand\Span{\operatorname{span}}

\def\det{\mathop{\hbox {det}}\nolimits}
\def\End{\mathop{\hbox {End}}\nolimits}

\newcommand{\pf}{\begin{proof}}
\newcommand{\epf}{\end{proof}}
\newcommand{\eq}{\begin{equation}}
\newcommand{\eeq}{\end{equation}}
\newcommand{\eqn}{\begin{equation*}}
\newcommand{\eeqn}{\end{equation*}}

\theoremstyle{plain}
\newtheorem{theorem}{Theorem}[section]
\newtheorem{cor}[theorem]{Corollary}
\newtheorem{prop}[theorem]{Proposition}
\newtheorem{lemma}[theorem]{Lemma}
\newtheorem{remark}[theorem]{Remark}

\newtheorem{definition}[theorem]{Definition}

\date{\today}
\title[Algebraic Dirac induction for discrete series of $SU(2, 1)$]{Algebraic Dirac induction for nonholomorphic discrete series of $SU(2, 1)$}
\author{Ana Prli\' c}\thanks{This work was partially supported by a grant from the Croatian Science Foundation.}

\subjclass[2010]{Primary 22E47; Secondary 22E46}

\keywords{Lie group, Lie algebra, discrete series, highest weight, minimal K–-type, Dirac operator, Dirac cohomology, Dirac induction}

\address{Department of Mathematics, University of Zagreb,
Bijeni\v cka cesta 30, 10000 Zagreb, Croatia.}
\email{anaprlic@math.hr}

\begin{document}

\begin{abstract}
In a joint paper P. Pand\v{z}i\' c and D. Renard proved that holomorphic and antiholomorphic discrete series representations can be constructed via algebraic Dirac induction. The group $SU(2,1)$, except for those two types, also has a third type of discrete series representations that are neither holomorphic nor antiholomorphic. In this paper we show that nonholomorphic discrete series representations of the group $SU(2,1)$ can also be constructed using algebraic Dirac induction.
\end{abstract}

\maketitle

\section{Introduction}
Dirac operators were introduced into representation theory by Partha\-sarathy in \cite{P} as a tool for constructing discrete series representations. An algebraic version of the Dirac operator was studied by D. Vogan [V].

Let $G$ be a connected real reductive Lie group with Cartan involution $\Theta$ such that $K = G^{\Theta}$ is a maximal compact subgroup of $G$ and let $\g = \ka \oplus \p$ be the Cartan decomposition of the complexified Lie algebra of $G$ corresponding to
$\Theta$. The Dirac operator $D$ is an element of the algebra $U(\g) \otimes C(\p)$, where $U(\g)$ denotes the universal enveloping
algebra of $\g$ and $C(\p)$ denotes the Clifford algebra of $\p$ with respect to the Killing form. For a $(\g, K)$--module $X$ Vogan defines Dirac cohomology as
$$
H_{V}^{D}(X) = \text{Ker}D/ \text{Im}D \cap \text{Ker}D.
$$
It is a $\tilde{K}$--module, where $\tilde{K}$ is the spin double cover of $K$. If $X$ is unitary, then $H_{V}^{D}(X) = \text{Ker}D = \text{Ker}D^2$.

The main result about Dirac cohomology was conjectured by Vogan \cite{V}, and proved by Huang and Pand\v{z}i\'{c} in \cite{HP}. Roughly speaking, it asserts that the Dirac cohomology, if nonzero, determines the infinitesimal character of the representation.
By now, Dirac cohomology has been calculated for many (unitary) modules (see \cite{DH}, \cite{HKP}, \cite{HPP}, \cite{BP1}, \cite{BP2}).

In \cite{PR} the authors describe certain constructions in the opposite direction which give representations with prescribed Dirac cohomology.  As the name fails to indicate, Dirac cohomology, as defined by Vogan, is not a cohomological theory. In particular, it is a functor which admits no adjoint. In \cite{PR}, two alternative definitions were proposed, which both coincide with the Vogan's definition for unitary and finite dimensional representations.
They are called Dirac cohomology and homology. The functor of the Dirac cohomology is left exact and admits a right adjoint, while the functor of the Dirac homology is right exact and admits a left adjoint. These adjoints are called Dirac induction functors.
One gets a representation with the Dirac cohomology $W$ by tensoring the algebra $U(\g) \otimes C(\p)$ with $W$ over a certain subalgebra. There are several options for the choice of the algebra over which we tensor. A minimal option is the subalgebra of $U(\g) \otimes C(\p)$ generated by the diagonal version of the Lie algebra of the group $K$ and by the Dirac operator. Modules obtained in this way are typically not irreducible. A maximal option is to include in the subalgebra all $K$--invariants in $U(\g) \otimes C(\p)$. A problem with this approach is that, in general, the algebra $(U(\g) \otimes C(\p))^{K}$ is complicated to describe. Therefore, an ``intermediate" option is also introduced, where we tensor over the subalgebra that does not include all $K$--invariants but only $K$--invariants in the algebra $C(\p)$ which is easy to describe. In \cite{PR}, the representations of holomorphic and antiholomorphic discrete series are constructed in this way.

The group $SU(2,1)$, except for holomorphic and antiholomorphic discrete series representations, also has a third type of discrete series which are neither highest weight modules nor lowest weight modules, the nonholomorphic discrete series representations. The irreducible unitary representations of the group $SU(n,1)$, $n \geq 2$ which are not principal series representations are considered in \cite{K}. Each of those representations is uniquely determined by its minimal $K$--type.

The algebra $(U(\g) \otimes C(\p))^{K}$ for the group $SU(2,1)$ is generated by five elements, see \cite{Pr}. Two of them are in the center $Z(\ka)$ of $U(\ka)$ - the Casimir element and the element spanning the center of $\ka$. One of the generators is in another abelian algebra, $C(\p)^K$ (which is three-dimensional in this case). The fourth generator is the Dirac operator, and the fifth generator is another distinguished element that can be thought of as a $\ka$-version of the Dirac operator. We will prove that the nonholomorphic discrete series representations can be obtained using  algebraic Dirac induction where we tensor over the algebra that is generated  by the Lie algebra of the group $K$ and by the whole algebra $(U(\g) \otimes C(\p))^{K}$.

The paper is organized as follows. In Section $2$, we give a brief review of the main definitions and results of \cite{PR}. In Section $3$, using results of \cite{K}, we describe nonholomorphic discrete series representation of $SU(2,1)$ in terms of the highest weight vectors of their minimal $K$--types. In Section $4$, we describe the $K$--module structure of the algebra $U(\g)$. Using this and the action of the algebra $(U(\g) \otimes C(\p))^{K}$ on the Dirac cohomology, we will reduce the obvious set of generators for the induced module to the set of generators which will later be shown to be a basis for the induced module. Finally, in Section $5$ we prove that nonholomorphic discrete series can be obtained via algebraic Dirac induction.

In future, we hope to generalize our results to nonholomorphic discrete series representations of the group $SU(n,1)$ and to some other more complicated examples.

\vspace{.2in}

\section{Dirac induction}
Let $\g_0$ be the Lie algebra of $G$ and let $B$ be an invariant nondegenerate symmetric bilinear form on $\g_{0}$. The Dirac operator $D = U(\g) \otimes C(\p)$ is given by
$$
D = \sum_{i} b_{i} \otimes d_{i},
$$
where $b_i$ is a basis of $\p$ and $d_i$ is the dual basis with respect to $B$. The operator $D$ is independent of the choice of basis
$b_i$ and $K$--invariant for the adjoint action on both factors. The adjoint action of $\ka$ on $\p$ defines a map $\text{ad}: \ka \longrightarrow \mathfrak{so}(\p)$. Composing it with the usual embedding of the Lie algebra $\mathfrak{so}(\p)$ into the Clifford algebra $C(\p)$, we get a Lie algebra map $\alpha: \ka \longrightarrow C(\p)$. Using $\alpha$ we can embed the Lie algebra $\ka$ diagonally into $U(\g) \otimes C(\p)$, by
$$
X \mapsto X_{\Delta} = X \otimes 1 + 1 \otimes \alpha(X).
$$
We will denote $\Delta(\ka)$ by $\ka_{\Delta}$. Let us denote $\mathcal{A} = U(\g) \otimes C(\p)$. Let $\mathcal{I}$ be the two-sided ideal in the algebra of $K$--invariants $\mathcal{A}^{K}$ generated by $D$ and let $\mathcal{B}$ be the $K$--invariant subalgebra of $\mathcal{A}$ with unit, generated by $\ka_{\Delta}$ and $\mathcal{I}$. Let us recall the definitions of the Dirac cohomology and homology from \cite{PR}.
\begin{definition}
Let $X$ be a $(\mathfrak{g}, K)$--module. Dirac cohomology of $X$ is the space of $\mathcal{I}$--invariants in $X \otimes S$. Therefore,
$$
H^{D}(X) = \{ v \in X \otimes S \, | \,  av = 0, \forall a \in \mathcal{I}\}.
$$
Dirac homology of $X$ is the space of $\mathcal{I}$--coinvariants in $X \otimes S$
$$
H_D(X) = X \otimes S / \mathcal{I}(X \otimes S).
$$
\end{definition}

In the most interesting cases when the module $X$ is finite-dimensional or unitary, $H_{D}(X)$, $H^{D}(X)$ and $H_{V}^{D}(X)$ all coincide and they are all equal to $\text{Ker}(D) = \text{Ker}(D^2)$.

The functor of Dirac cohomology $H^{D} : \mathcal{M}(\mathcal{A}, \tilde{K}) \longrightarrow \mathcal{M}(\ka_{\Delta}, \tilde{K})$ has a left adjoint, the functor
$$
\text{Ind}_D : W \mapsto \mathcal{A} \otimes_{\mathcal{B}}W.
$$
The functor of Dirac homology $H_{D} : \mathcal{M}(\mathcal{A}, \tilde{K}) \longrightarrow \mathcal{M}(\ka_{\Delta}, \tilde{K})$ has a right adjoint, the functor
$$
\text{Ind}^{D} : W \mapsto \text{Hom}_{\mathcal{B}}(\mathcal{A}, W)_{\tilde{K}\text{--finite}}.
$$

Furthermore, $H^{D}(X)$ and $H_D(X)$ consist of full $\tilde{K}$--isotypic components of $X \otimes S$, which does not seem to be true for $H_{V}^{D}(X)$.

Modules we get by the first version of induction are tipically not irreducible. Therefore, we would like to tensor over a ``bigger" algebra $\mathcal{B}$. The algebra of $K$--invariants, $\mathcal{A}^{K}$, acts on each $\tilde{K}$--isotypic component of $X \otimes S$. Furthermore, from Theorem \cite[Theorem~4.10.]{PR} follows that the action of the algebra $\mathcal{A}^{K}$ on any  nontrivial $\tilde{K}$--isotypic component of $X \otimes S$ determines the irreducible $(\mathcal{A}, \tilde{K})$--module $X \otimes S$ up to isomorphism. This is a version of a theorem of Harish-Chandra (see \cite{HC}, \cite{LMC}).

Now we can consider $H^{D}$ and $H_{D}$ as functors from the category of $(\mathcal{A}, \tilde{K})$--modules to the category of ($U(\ka_{\Delta}) \mathcal{A}^{K}, \tilde{K})$--modules on which $\mathcal{I}$ acts by zero. Then the functor $\widehat{\text{Ind}}_D = \mathcal{A} \otimes_{U(\ka_{\Delta}) \mathcal{A}^{K}} \cdot $ is left adjoint to the functor $H^{D}$, while the functor $\widehat{\text{Ind}}^{D} = \text{Hom}_{U(\ka_{\Delta}) \mathcal{A}^{K}}(\mathcal{A}, \cdot)_{\tilde{K}\text{--finite}}$ is right adjoint to the functor $H_D$. The problem with this approach is that the algebra $\mathcal{A}^{K}$ contains $U(\g)^{K}$ which is in general very hard to describe. The same can be expected for the algebra $\mathcal{A}^{K}$.

Fortunately, the algebra of $K$--invariants contains the algebra $C(\p)^{K}$ which is very easy to describe. It can be easily seen (see \cite{PR}) that $C(\p)^{K} \simeq \End_{\tilde{K}}(S)$ which is spanned by the projections onto $\tilde{K}$--types in $S$. Moreover, on each irreducible $C(\p)^{K}$--module exactly one of the projections acts by $1$ and all other act by $0$.

Therefore, the third approach is to consider an ``intermediate" version of induction where we tensor over the algebra $U(\ka_{\Delta})(C(\p)^{K} + \mathcal{I})$.

Let $W$ be a $U(\ka_{\Delta})$--module. Let $U(\ka_{\Delta})(C(\p)^{K} + \mathcal{I})$ act on $W$ so that $\mathcal{I}$ acts by zero, one of the projections $p_1, p_2, \cdots, p_m$ acts by $1$, and the rest of them by zero. Then we define
\begin{align*}
\widetilde{\text{Ind}}_{D}(W) & = \mathcal{A} \otimes_{U(\ka_{\Delta})(C(\p)^{K} + \mathcal{I})} W \\
\widetilde{\text{Ind}}^{D}(W) & = \text{Hom}_{U(\ka_{\Delta})(C(\p)^{K} + \mathcal{I})}(\mathcal{A}, W)_{\widetilde{K}\text{--finite}}.
\end{align*}
The functor $\widetilde{\text{Ind}}_D = \mathcal{A} \otimes_{U(\ka_{\Delta}) (C(\p)^{K} + \mathcal{I})} \cdot$ is left adjoint to the  functor $H^{D}$ considered as a functor from the category of $(\mathcal{A}, \widetilde{K})$--modules to the category of  ($U(\ka_{\Delta})(C(\p)^{K} + \mathcal{I}), \widetilde{K})$--modules on which $\mathcal{I}$ acts by zero, while the functor $\widetilde{\text{Ind}}^D = \text{Hom}_{U(\ka_{\Delta}) (C(\p)^{K} + \mathcal{I})}(\mathcal{A}, \cdot)_{\tilde{K}\text{--finite}}$ is right adjoint to the functor $H_D$. All holomorphic discrete series representations can be constructed via intermediate version of Dirac induction.

\vspace{.2in}

\section{Nonholomorphic discrete series representations of the group $SU(2,1)$}

We will denote by $G$ the Lie group
$$
SU(2,1) = \{ g \in GL(3, \mathbb{C}) \, | \, \det g = 1, \, g^{*} \Gamma g = \Gamma \},
$$
where $\Gamma = \text{diag} (1, 1, -1)$. Its Cartan involution is given by $\Theta(g) = (g^{-1})^{*}$ and the corresponding maximal compact subgroup is
$$
K =  \{ g \in SU(2,1) \, | \, \Theta(g) = g \} = \left \{ \left [ \begin{array}{cc} A &  \\  & (\det(A))^{-1} \end{array} \right ] \, \bigg| \, A \in U(2) \right \} \cong U(2).
$$
The Lie algebra $\mathfrak{g_0}$ of the group $SU(2,1)$ is given by
$$
\mathfrak{su}(2,1) = \{ x \in \mathfrak{gl}(3, \mathbb{C}) \, \big| \, \text{tr}(X) = 0, x^{*} \Gamma + \Gamma x = 0 \}.
$$
Its complexification is $\mathfrak{g} \cong sl(3, \mathbb{C})$. The  complexified Lie algebra $\mathfrak{k}$ of the Lie algebra of $K$ is isomorphic to the Lie algebra $\mathfrak{gl}(2, \mathbb{C})$. Let $\g = \ka \oplus \p$ be the Cartan decomposition of the Lie algebra $\mathfrak{g}$ corresponding to the Cartan involution $\theta(x) = - x^{*}$. Then
$$
\mathfrak{k} = \text{span}_{\mathbb{C}} \{ H_1, H_2, E, F \}, \quad \mathfrak{p} = \text{span}_{\mathbb{C}} \{ E_1, E_2, F_1, F_2 \},
$$
where $H_1 = e_{11} - e_{33}$, $H_2 = e_{22} - e_{33}$, $E = e_{12}, F = e_{21}, E_{i} = e_{i3}, F_i = e_{3i}$ for $i = 1, 2$. Here $e_{ij}$
denotes the matrix in $\mathfrak{g}$  with $ij$ entry equal to one and all other entries equal to $0$.

The Cartan subalgebra $\mathfrak{h}_{0}$ of $\mathfrak{g}_{0}$ (and $\mathfrak{k}_{0}$) consists of diagonal matrices in $\mathfrak{g}_{0}$.

We identify $\mathfrak{h}^{*}$ with the set
$$
\mathbb{C}_{0}^{3} = \{ s \in \mathbb{C}^{3} \, \big| \, s_1 + s_2 + s_3 = 0 \},
$$
where $s \in \mathbb{C}_{0}^{3}$ corresponds to the functional on $\h$ given by $H_j \mapsto s_j - s_3, j=1,2$. The root system of $(\mathfrak{g}, \mathfrak{h})$ is
$$
R = \{ \epsilon_j - \epsilon_k \, \big| \, 1 \leq j,k \leq 3, j \neq k \},
$$
where $\epsilon_i$ denotes the functional $h \mapsto h_i$ on $\h$. Furthermore, the sets of compact and noncompact roots are given by
$$
R_{K} = \{ \epsilon_1 - \epsilon_2, \epsilon_2 - \epsilon_1 \}, \quad R_{P} = \{ \epsilon_1 - \epsilon_3, \epsilon_2 - \epsilon_3, \epsilon_3 - \epsilon_1, \epsilon_3 - \epsilon_2 \}.
$$

The Weyl group is the group of permutations of coordinates, i.e. $W(R) = S_3$. We choose the $R_K$-- Weyl  chamber to be
$C = \{ s \in \mathbb{R}_{0}^{3} \, | \, s_1 > s_2 \}$.
The  corresponding set of positive compact roots is given by $R_{K}^{C} = \{ \epsilon_1 - \epsilon_2 \}$. There are three $R$-- Weyl chambers contained in the $R_K$--Weyl
chamber $C$, and they are:
\begin{align*}
D_0 & = \{ s \in \mathbb{R}_{0}^{3} \, | \, s_3 > s_1 > s_2 \} \\
D_1 & = \{ s \in \mathbb{R}_{0}^{3} \, | \, s_1 > s_3 > s_2 \} \\
D_2 & = \{ s \in \mathbb{R}_{0}^{3} \, | \, s_1 > s_2 > s_3 \}.
\end{align*}
The corresponding noncompact positive roots are given by
\begin{align*}
R_{P}^{D_0} & = R^{D_0} \setminus R_{K}^{C} = \{ \epsilon_3 - \epsilon_1, \epsilon_3 - \epsilon_2 \} \\
R_{P}^{D_1} & = R^{D_1} \setminus R_{K}^{C} = \{ \epsilon_1 - \epsilon_3, \epsilon_3 - \epsilon_2 \} \\
R_{P}^{D_2} & = R^{D_2} \setminus R_{K}^{C} = \{ \epsilon_1 - \epsilon_3, \epsilon_2 - \epsilon_3 \}.
\end{align*}

Let $\hat{K}$ denote the set of equivalence classes of irreducible representations of $K$. In \cite{K} $\hat{K}$ is identified with
$$
\{ (q_1, q_2, -q_1 - q_2) \, | \, q_1, q_2 \in \frac{1}{3} \mathbb{Z}, q_1 - q_2 \in \mathbb{Z}_{+}\}
$$

For each $p \in \hat{K}$ there is a unique irreducible unitary representation of $G$ which is not a principal series representation with minimal $K$--type $p$. We will denote that representation by $\pi_p$. It is also proved that the representation $\pi_p$ belongs to a discrete series representation if and only if  $p + \rho_K - \rho_{P}^{D_j} \in D_j$, for some $j \in \{ 0, 1, 2 \}$. Here $\rho_K$ is the half sum of compact positive roots and $\rho_{P}^{D_j}$  is the half sum of noncompact positive roots $R_{P}^{D_j}$, $j \in \{ 0, 1, 2 \}$. If  $p + \rho_K - \rho_{P}^{D_2} \in D_2$, then $\pi_p$ is a lowest weight module i.e., belongs to the holomorphic discrete series.  If  $p + \rho_K - \rho_{P}^{D_0} \in D_0$, then $\pi_p$ is a highest weight module i.e., belongs to the anti-holomorphic discrete series. In the case when $p + \rho_K - \rho_{P}^{D_1} \in D_1$, $\pi_p$ is neither a highest nor a lowest weight module and its $K$--spectrum is given by $$\Gamma(\pi_p) = \{ q \in \hat{K} \, | \, q_1 \in \{ p_1, p_1 + 1, p_1 + 2, \cdots\}, q_2 \in \{ p_2, p_2 - 1, p_2 - 2, \cdots \}\}.$$

Let us denote $h_1 = \frac{2 H_1 - H_2}{3}$ and $h_2 = \frac{2 H_2 - H_1}{3}$. For $(q_1, q_2) \in (\frac{1}{3} \mathbb{Z})^{2}$ and $q_1 - q_2 \in \mathbb{Z}_{+}$, the irreducible representation of the group $K$ with the highest weight $(q_1, q_2)$ is the representation with the basis $\{ v_{q_1, q_2}^{q_1 - q_2 - 2s} \, | \, s \in \{0, 1, \cdots, q_1 - q_2 \} \}$ and the following action of $\mathfrak{k}$:
\begin{align}\label{kaction}
h_1 v_{q_1, q_2}^{q_1 - q_2 - 2s} & = (q_1 - s)v_{q_1, q_2}^{q_1 - q_2 - 2s}, \text{for  } s \in \{ 0, 1, \cdots, q_1 - q_2 \} \\
h_2 v_{q_1, q_2}^{q_1 - q_2 - 2s} & = (q_2 + s) v_{q_1, q_2}^{q_1 - q_2 - 2s}, \text{for } s \in \{ 0, 1, \cdots, q_1 - q_2 \} \notag \\
E v_{q_1, q_2}^{q_1 - q_2 - 2s} & = s \cdot (q_1 - q_2 - s + 1)v_{q_1, q_2}^{q_1 - q_2 - 2(s-1)}, \text{for  } s \in \{1, \cdots, q_1 - q_2 \} \notag \\
F v_{q_1, q_2}^{q_1 - q_2 - 2s} & = v_{q_1, q_2}^{q_1 - q_2 - 2(s+1)} \text{for } s \in \{ 0, 1, \cdots, q_1 - q_2 - 1 \} \notag \\
E v_{q_1, q_2}^{q_1 - q_2} & = 0 \notag \\
F v_{q_1, q_2}^{-q_1 + q_2} & = 0. \notag
\end{align}
Furthermore, if $\pi_p$ is a nonholomorphic discrete series representation, then it follows from \cite{K}  that there exists a basis
$$ \{ v_{q_1, q_2}^{q_1 - q_2 - 2s} \, | \, (q_1, q_2) \in \Gamma(\pi_p), s \in \{0, 1, \cdots, q_1 - q_2 \} \}$$ such that
\begin{align}\label{paction}
E_1 v_{q_1, q_2}^{q_1 - q_2} &= A_{1}^{\pi_p}(q) v_{q_1 + 1, q_2}^{(q_1 + 1) - q_2} \\
F_2 v_{q_1, q_2}^{q_1 - q_2} &= B_{2}^{\pi_p}(q) v_{q_1, q_2 - 1}^{q_1 - (q_2 - 1)} \notag,
\end{align}
where $A_{1}^{\pi_p}(q) \neq 0$ and $B_{2}^{\pi_p}(q) \neq 0$ for each $q = (q_1, q_2) \in \Gamma(\pi_p)$.
From this, it easily follows
\begin{theorem}\label{basis1}
If $\pi_p$ is nonholomorphic discrete series representation with representation space $X$, and if
$$\{ v_{q_1, q_2}^{q_1 - q_2 - 2s} \, | \, (q_1, q_2) \in \Gamma(\pi_p), s \in \{ 0,1, \cdots, q_1 - q_2 \} \}$$ is a basis for $X$ as in (\ref{paction}), then
$$
\{ F^{t} E_{1}^{n} F_2^{m} v_{p_1, p_2}^{p_1 - p_2} \, | \, n, m \in \mathbb{N}_0, t \in \{ 0, 1, \cdots, (p_1 + n) - (p_2 - m)\}\}.
$$
is a basis for $X$.
\end{theorem}

\section{Generating system for the induced module}
We will consider reduced version of the algebraic Dirac induction where we tensor over the algebra $\mathcal{B}$  generated by the algebra $\mathfrak{k}_{\Delta}$ and the algebra $\mathcal{A}^K$. It follows from \cite{Pr} that the algebra $\mathcal{A}^{K}$ is generated by $Z(\ka)$, $C(\p)^K$, by the Dirac operator and by the element
\begin{align}\label{k-Dirac}
D^{\ka} & = E \otimes \alpha(F) + \frac{1}{2}(h_1 - h_2) \otimes \alpha(h_1 - h_2) \\ & + \frac{3}{2}(h_1 + h_2) \otimes \alpha(h_1 + h_2) +  F \otimes \alpha(E) \notag \\
& = - \frac{1}{4}(2 E \otimes E_2 F_1 + (h_1 - h_2) \otimes (E_1 F_1 - E_2 F_2) + \notag \\
& 3(h_1 + h_2) \otimes (E_1 F_1 + E_2 F_2) + 2 F \otimes E_1 F_2) - \frac{3}{2}(h_1 + h_2) \otimes 1 \notag.
\end{align}
\begin{remark}
Since the basis $\left(E, F, h_1 - h_2, h_1 + h_2 \right)$ of $\ka$ is dual to the basis $\left(F, E, \frac{1}{2}(h_1 - h_2), \frac{3}{2}(h_1 + h_2) \right)$, the element \eqref{k-Dirac} can be thought of as a $\ka$-version of the Dirac operator. We will call it the $\ka$--Dirac.
\end{remark}
We will give one generating system for $(\mathcal{A}, \tilde{K})$--module $\text{Ind}_{D}(W) = \mathcal{A} \otimes_{\mathcal{B}} W $, where $W = H^{D}(X)$.

One can easily check that
$$
W = \text{span}_{\mathbb{C}}\{ v_{p_1, p_2}^{p_1 - p_2 - 2s}\otimes E_1 - (p_1 - p_2 - (s-1)) v_{p_1, p_2}^{p_1 - p_2 - 2(s-1)} \otimes E_2 \, \big| \, s \in \{ 1,2, \cdots, p_1 - p_2\} \}.
$$

Let us denote
$$
w_s = v_{p_1, p_2}^{p_1 - p_2 - 2s} \otimes E_1 - (p_1 - p_2 - (s-1)) v_{p_1, p_2}^{p_1 - p_2 - 2(s-1)} \otimes E_2, \, s  \in \{ 1, 2, \cdots, p_1 - p_2 \}.
$$
Then we have
\begin{align}\label{dircoho}
h_{1_{\Delta}}w_s & = (p_1 - s + \frac{1}{2}) w_s \notag \\
h_{2_{\Delta}}w_s & = (p_2 + s - \frac{1}{2}) w_s \notag \\
E_{\Delta}w_{s} & = (s-1)(p_1 - p_2 - (s-1)) w_{s - 1} \notag \\
F_{\Delta}w_s & = w_{s+1}.
\end{align}

Let
$$
Z = \text{span}_{\mathbb{C}} \{ ab \otimes w - a \otimes bw \, | \, a \in \mathcal{A}, b \in \mathcal{B}, w \in W \}.
$$

We are going to reduce the obvious generating system for $\mathcal{A} \otimes_{\mathcal{B}}W$ given by
$$
\{ a \otimes w + Z | a \in \mathcal{A}, w \in W \}.
$$
The first step is to ``remove" $U(\ka)$.

For $x \in \mathfrak{k}$, $y \in C(\mathfrak{p})$ and $w \in W$ we have
\begin{align*}
(x \otimes y) \otimes w & =\left((1 \otimes y) \cdot (x_{\Delta} - 1 \otimes \alpha(x)) \right) \otimes w = \\
& = ((1 \otimes y) \cdot x_{\Delta}) \otimes w - (1 \otimes y \cdot \alpha(x)) \otimes w \\
& = \underbrace{((1 \otimes y) \cdot x_{\Delta}) \otimes w - (1 \otimes y) \otimes x_{\Delta}w}_{\in Z} + (1 \otimes y) \otimes x_{\Delta}w \\
& - (1 \otimes y \cdot \alpha(x)) \otimes w \\
& \in \left(1 \otimes C(\mathfrak{p}) \right) \otimes W + Z.
\end{align*}

It follows:
\begin{equation}\label{elimk}
\left(U(\mathfrak{k}) \otimes C(\mathfrak{p})\right) \otimes W \subset \left( 1 \otimes C(\mathfrak{p}) \right) \otimes W + Z.
\end{equation}

Now, we are going to describe the structure of the algebra $U(\g)$.

\begin{prop}\label{ug}
We have
\begin{align*} U(\g) = \Span_{\C}\{ & x(E_1 F_1 + E_2 F_2)^{t}y \, | \, x \in \{ E_{1}^{n}F_{2}^{m}, (\text{ad}F)(E_{1}^{n}F_{2}^{m}), \cdots, \\
& (\text{ad}F)^{n+m} (E_{1}^{n} F_{2}^{m}) \, | \, n, m \in \mathbb{N}_{0}\} , t \in \mathbb{N}_{0}, y \in U(\ka) \}.
\end{align*}
\end{prop}

\proof
Let us introduce the following notation
\begin{align*} T = \text{span}_{\C}\{ & x(E_1 F_1 + E_2 F_2)^{t}y \, | \, x \in \{ E_{1}^{n}F_{2}^{m}, (\text{ad}F)(E_{1}^{n}F_{2}^{m}), \cdots, \\
& (\text{ad}F)^{n+m} (E_{1}^{n} F_{2}^{m}) \, | \, n, m \in \mathbb{N}_{0}\} , t \in \mathbb{N}_{0}, y \in U(\ka) \}
\end{align*}
We will show by induction that $U_{l}(\g) \subset T$ for all $l \in \mathbb{N}_{0}$. The claim is obvious for $l = 0$ and $l = 1$. Let us assume that the claim is true for $l \in \mathbb{N}_{0}$ less than or equal to some fixed $k \geq 1$. By Poincare-Birkhoff-Witt's theorem, it is enough to prove that elements of the form $E_{1}^{p} E_{2}^{q} F_{1}^{r} F_{2}^{s}$, $p,q,r,s \in \mathbb{N}_{0}, p+q+r+s = k+1$ are in $T$. From \cite[Lemma~2.3.]{Pr} it follows that
$$
S^{n}(\mathfrak{p}) = \left( \sum_{i = 0}^{n} \oplus V_{(n - i, -i)} \right) \oplus (E_1 F_1 + E_2 F_2) S^{n-2}(\mathfrak{p}),
$$
for all $n \in \mathbb{N}, n \geq 2$, where $V_{(n-i, -i)}$ is the irreducible $\mathfrak{k}$--module with the highest weight vector $E_{1}^{n-i} F_{2}^{i}$ and the highest weight $(n - i, -i)$. Now we have
\begin{align*}
&E_{1}^{p} E_{2}^{q} F_{1}^{r} F_{2}^{s}  \in \text{span}_{\C}\{ E_{1}^{k+1-i}F_{2}^{i}, (\text{ad}F)(E_{1}^{k+1-i}F_{2}^{i}), \cdots, (\text{ad}F)^{k+1} (E_{1}^{k+1-i} F_{2}^{i}) \, | \\
& \, 0 \leq i \leq k+1 \} \oplus  \text{span}_{\C} \{ E_{1}^{a+1} E_{2}^{b} F_{1}^{c+1} F_{2}^{d} + E_{1}^{a} E_{2}^{b+1} F_{1}^{c} F_{2}^{d+1} \, | \, a, b, c, d \in \mathbb{N}_{0}, \\
& a + b + c + d = k-1 \}.
\end{align*}
Therefore, it is enough to show that
$
E_{1}^{a+1} E_{2}^{b} F_{1}^{c+1} F_{2}^{d} + E_{1}^{a} E_{2}^{b+1} F_{1}^{c} F_{2}^{d+1} \in T$ for $a, b, c, d \in \mathbb{N}_{0}$, $a + b + c + d = k-1$.
By induction on $c$ and $d$ it can be easily seen that
\begin{align*}
E_{1}^{a+1} E_{2}^{b} F_{1}^{c+1} F_{2}^{d} + E_{1}^{a} E_{2}^{b+1} F_{1}^{c} F_{2}^{d+1}  & \in E_{1}^{a} E_{2}^{b} F_{1}^{c} F_{2}^{d}(E_1 F_1 + E_2 F_2) + U_{k}(\g) \\
& \subset U_{k-1}(\g)(E_1 F_1 + E_2 F_2) + U_{k}(\g).
\end{align*}
Since $E_1 F_1 + E_2 F_2$ commutes with $\ka$, the proposition follows.
\qed

\bigskip

Let us denote  $C = E_1 \otimes F_1 + E_2 \otimes F_2$ and $C^{-} = F_1 \otimes E_1 + F_2 \otimes E_2$.
\begin{lemma}\label{halfdirac}
The elements $C$ and $C^{-}$ are contained in the ideal $\mathcal{I}$.
\end{lemma}
\begin{proof}
We have $D = C + C^{-}$. Furthermore, it can easily be seen that
$$
D (-\frac{1}{2} \otimes (E_1 F_1 + E_2 F_2)) - (-\frac{1}{2} \otimes (E_1 F_1 + E_2 F_2)) D = C - C^{-}.
$$
Since $-\frac{1}{2} \otimes (E_1 F_1 + E_2 F_2) \in \mathcal{A}^{K}$, we have
\begin{align*}
C & = \frac{1}{2}(D + D (-\frac{1}{2} \otimes (E_1 F_1 + E_2 F_2)) - (-\frac{1}{2} \otimes (E_1 F_1 + E_2 F_2)) D) \in \mathcal{I}\\
C^{-} & = D - C \in \mathcal{I}.
\end{align*}
\end{proof}
Since the ideal $I$ acts on $W$ by zero, we get $CC^{-} \otimes W \subset Z$ and $C^{-}C \otimes W \subset Z$. Furthermore, we have
\begin{align}\label{connection}
(E_1 F_1 + E_2 F_2) \otimes 1 & = -\frac{1}{4} (2 CC^{-} + 2 C^{-}C  + 3(h_1 + h_2) \otimes (E_1 F_1 + E_2 F_2) \\
& + 2 E \otimes E_2 F_1 + (h_1 - h_2) \otimes (E_1 F_1 - E_2 F_2) + 2 F \otimes E_1 F_2). \notag
\end{align}
from where it follows that
$$
((E_1 F_1 + E_2 F_2) \otimes 1) \otimes w \in \left( U(\mathfrak{k}) \otimes C(\mathfrak{p}) \right) \otimes W + Z.
$$
By \eqref{elimk}, proposition \ref{ug} and \eqref{connection} we see that the vector space $\mathcal{A} \otimes_{\mathcal{B}} W$ is spanned by the elements of the form

\begin{equation}\label{ug_reduced}
(x \otimes y) \otimes w, \quad x \in V_{(n, -m)}, \, y \in C(\mathfrak{p}), \, w \in W, \, n,m \in \mathbb{N}_{0}.
\end{equation}

\begin{remark}
From \eqref{connection}, \eqref{k-Dirac}, and from the identity
\begin{align*}
(h_1 + h_2) \otimes 1 & = (h_1 + h_2)_{\Delta} - 1 \otimes \alpha(h_1 + h_2) \\
& = (h_1 + h_2)_{\Delta} - 1 \otimes (-\frac{1}{2}(E_1 F_1 + E_2 F_2) - 1) \in \mathcal{B},
\end{align*}
it follows that $(E_1 F_1 + E_2 F_2) \otimes 1 \in \mathcal{B}$.
\end{remark}

The next step is to reduce a part of the algebra $C(\p)$.

\begin{lemma}\label{E1E2}
If $T \in \End_{\mathbb{C}}(S) \cong C(\mathfrak{p})$ is any linear operator such that
$$
T(E_1) = T(E_2) = 0,
$$
then $T = T p^{'}$, where $p^{'} = p_1 + p_3$ is the sum of the projections on $\Span_{\mathbb{C}} \{ 1 \}$, respectively $\Span_{\mathbb{C}} \{ E_1 \wedge E_2 \}$.
\end{lemma}

\proof
For $s \in \text{span}_{\mathbb{C}} \{ 1, E_1 \wedge E_2 \}$ we have
$p'(s) = s$ and then $Tp'(s) = T(s)$. Furthermore, $T(E_1) = T(E_2) = 0 = T p'(E_1) = T p'(E_2)$. Since $S =  \text{span}_{\mathbb{C}}\{1, E_1, E_2, E_1 \wedge E_2 \}$ it follows that $T = T p'$.
\qed

\bigskip

In the spin module $S$, the following identities hold
\begin{align*}
E_1 E_2 \cdot E_1 & = E_1 \cdot (-E_1 \wedge E_2) = 0, \quad & F_1 F_2 \cdot E_1 & = F_1 \cdot 0 = 0, \\
E_1 E_2 \cdot E_2 & = E_1 \cdot 0 = 0,  \quad & F_1 F_2 \cdot E_2 & = F_1 \cdot (-2) = 0.
\end{align*}
 By lemma \ref{E1E2}, we have $E_1 E_2 p' = E_1 E_2$ and $F_1 F_2 p' = F_1 F_2$. Since $1 \otimes p_1$ and $1 \otimes p_3$ act on $W$ by zero, we have $(1 \otimes p')w = 0$ for $w \in W$. For $a \in \mathcal{A}$ and $s \in \text{span}_{\mathbb{C}}\{ E_1 E_2, F_1 F_2\}$ we get
\begin{align*}
\left( a(1 \otimes s) \right) \otimes w & = (a(1 \otimes sp')) \otimes w  = (a(1 \otimes s)(1 \otimes p')) \otimes w \\
& = (a (1 \otimes s)(1 \otimes p')) \otimes w - (a(1 \otimes s)) \otimes \underbrace{(1 \otimes p')w}_{0} \in Z.
\end{align*}
From here it follows that
\begin{equation}\label{pp}
(U(\mathfrak{g}) \otimes C(\mathfrak{p})\p^{-}\p^{-}) \otimes W \subset Z \text{ and } (U(\mathfrak{g}) \otimes C(\mathfrak{p})\p^{+}\p^{+}) \otimes W \subset Z.
\end{equation}

Furthermore, since the projection on the two-dimensional component in the spin module $S$ is equal to $p_2 = -\frac{1}{2}(E_1 F_1 + E_2 F_2)$ and since $1 \otimes p_2$ acts on $W$ by one, then the element $1 \otimes (E_1 F_1 + E_2 F_2) \in \mathcal{B}$ acts on $W$ by $-2$.

Therefore, for $a \in U(\g)$ and $w \in W$
we have
\begin{align}\label{elim_one}
& (a \otimes 1) \otimes w  = -\frac{1}{2}(a \otimes 1) \otimes (-2w) = - \frac{1}{2}(a \otimes 1) \otimes ((1 \otimes (E_1 F_1 + E_2 F_2)) \cdot w) \\
& = \frac{1}{2}	\underbrace{\left( (a \otimes 1)(1 \otimes (E_1 F_1 + E_2 F_2)) \otimes w - (a \otimes 1) \otimes ((1 \otimes (E_1 F_1 + E_2 F_2)) \cdot w \right)}_{\in Z}  \notag \\
& -\frac{1}{2}(a \otimes (E_1F_1 + E_2 F_2)) \otimes w \notag \\
& \in -\frac{1}{2}(a \otimes E_1 F_1) \otimes w - \frac{1}{2}(a \otimes E_2 F_2) \otimes w + Z. \notag
\end{align}

From \eqref{ug_reduced}, \eqref{pp} and \eqref{elim_one} it follows that the vector space $\mathcal{A} \otimes_{\mathcal{B}} W$ is spanned by the set
\begin{align}\label{cp_reduced}
\{ (x \otimes y) \otimes w \, | \,  & x \in \{ E_{1}^{n} F_{2}^{m}, (\text{ad}F)(E_{1}^{n} F_{2}^{m}), \cdots, (\text{ad}F)^{n+m}(E_{1}^{n} F_{2}^{m}) \, | \, n,m \in \mathbb{N}_0 \}, \\
& y \in \{ E_1, E_2, F_1, F_2, E_1 F_1, E_1 F_2, E_2 F_1, E_2 F_2 \}, w \in W \}. \notag
\end{align}

\bigskip

Furthermore, one can easily check that
$$
(F_1 \otimes E_1 + F_2 \otimes E_2)(1 \otimes F_1 F_2) - (1 \otimes F_1 F_2)(F_1 \otimes E_1 + F_2 \otimes E_2) = - 2(F_1 \otimes F_2 - F_2 \otimes F_1)
$$
and
$$
(E_1 \otimes F_1 + E_2 \otimes F_2)(1 \otimes E_1 E_2) - (1 \otimes E_1 E_2)(E_1 \otimes F_1 + E_2 \otimes F_2) = - 2(E_1 \otimes E_2 - E_2 \otimes E_1).
$$
From here, \eqref{pp} and $C \otimes W \subset Z$, $C^{-} \otimes W \subset Z$ we get
\begin{equation}\label{anticom}
(E_1 \otimes E_2 - E_2 \otimes E_1) \otimes w \in Z \text{ and } (F_1 \otimes F_2 - F_2 \otimes F_1) \otimes w \in Z.
\end{equation}

A straightforward calculation shows that the element $(E_1 F_1 + E_2 F_2) \otimes 1$ acts on $W$ by $(p_1 + 2 p_2 - 1)$.

\begin{prop}\label{F}
For $s \in \{1,2, \cdots, p_1 - p_2 \}$ we have
$$(1 \otimes F_1) \otimes (p_2 - p_1 + s)w_s \in (1 \otimes F_2) \otimes F_{\Delta}w_s + Z.$$
\end{prop}

\proof
For $w \in W$ we have
\begin{equation}\label{1o}
F_{\Delta}(1 \otimes F_2) \otimes w 
 \in (1 \otimes F_2) \otimes F_{\Delta}w - (1 \otimes F_1) \otimes w + Z.
\end{equation}
Furthermore, using \eqref{pp} we get
\begin{equation}\label{2o}
F_{\Delta}(1 \otimes F_2) \otimes w  
\in (F \otimes F_2) \otimes w + Z.
\end{equation}
From \eqref{anticom}, $C \otimes W \subset Z$ and $[E_1, F_1] = F$ we get
\begin{align}\label{3o}
(F \otimes F_2) \otimes w & \in ((E_2 F_2 + F_1 E_1) \otimes F_1) \otimes w + Z \\
& = ((E_1 F_1 + E_2 F_2) \otimes F_1) \otimes w + ((F_1 E_1 - E_1 F_1) \otimes F_1) \otimes w + Z \notag \\
& = ((E_1 F_1 + E_2 F_2) \otimes F_1) \otimes w- (H_1 \otimes F_1) \otimes w + Z. \notag
\end{align}

Since $H_1 = 2 h_1 + h_2$, $\alpha(h_1) = - \frac{1}{2} - \frac{1}{2}E_1 F_1$ and $\alpha(h_2) = - \frac{1}{2} - \frac{1}{2}E_2 F_2$, we have $\alpha(H_1) = - E_1F_1 - \frac{1}{2}E_2 F_2 - \frac{3}{2}$. Using \eqref{pp} we get
\begin{align}\label{4o}
(H_1 \otimes F_1) \otimes w & = ((1 \otimes F_1)(H_1 \otimes 1 + 1 \otimes \alpha(H_1))) \otimes w - (1 \otimes F_1\alpha(H_1)) \otimes w \\
& \in (1 \otimes F_1) \otimes H_{1_{\Delta}}w - \frac{1}{2}(1 \otimes F_1) \otimes w + Z. \notag
\end{align}
Now from \eqref{1o}, \eqref{2o}, \eqref{3o} and \eqref{4o} it follows that
\begin{align*}
(1 \otimes F_1) \otimes w & \in (1 \otimes F_2) \otimes F_{\Delta} w - F_{\Delta}(1 \otimes F_2) \otimes w + Z \\
& = (1 \otimes F_2) \otimes F_{\Delta}w - ((E_1 F_1 + E_2 F_2) \otimes F_1) \otimes w \\
& + (1 \otimes F_1) \otimes H_{1_{\Delta}}w - \frac{1}{2}(1 \otimes F_1) \otimes w + Z.
\end{align*}
From here we get
\begin{align}\label{5o}
& (1 \otimes F_1) \otimes (\frac{3}{2}w - H_{1_{\Delta}} w) \in (1\otimes F_2) \otimes F_{\Delta}w - ((E_1 F_1 + E_2 F_2) \otimes F_1) \otimes w + Z \\
& = (1\otimes F_2) \otimes F_{\Delta}w - (1 \otimes F_1) \otimes ((E_1 F_1 + E_2 F_2) \otimes 1)w \notag \\
& - \underbrace{(((E_1 F_1 + E_2 F_2) \otimes F_1) \otimes w- (1 \otimes F_1) \otimes ((E_1 F_1 + E_2 F_2) \otimes 1)w)}_{\in Z} + Z. \notag
\end{align}
Since the element $(E_1 F_1 + E_2 F_2) \otimes 1 \in \mathcal{B}$ acts on $W$ by the scalar $(p_1 + 2p_2 - 1)$ and $H_{1_\Delta} w_s = (2h_{1_\Delta} + h_{2_\Delta})w_s = (2p_1 + p_2 - s + \frac{1}{2})w_s$ (see \eqref{dircoho}), the \eqref{5o} implies
\begin{equation*}
(1 \otimes F_1) \otimes (p_2 - p_1 + s)w_s \in (1 \otimes F_2) \otimes F_{\Delta}w_s + Z \text{ for } s \in \{ 1,2, \cdots, p_1 - p_2 \}.
\end{equation*}
This finishes the proof.
\qed

\begin{cor}\label{E}
For $s \in \{ 1,2, \cdots, p_1 - p_2\}$ we have
\begin{align*}
(1 \otimes E_1 F_1) \otimes (p_2 - p_1 + s) w_s & \in  (1 \otimes E_1 F_2) \otimes F_{\Delta}w_s + Z \\
(1 \otimes E_2 F_1) \otimes (p_2 - p_1 + s) w_s & \in  (1 \otimes E_2 F_2) \otimes F_{\Delta}w_s + Z \\
(1 \otimes E_2) \otimes (p_2 - p_1 + s) w_s & \in - (1 \otimes E_1) \otimes F_{\Delta}w_s + Z
\end{align*}
\end{cor}

\proof
The first two claims follow from the previous lemma. Furthermore, we have
$$
(1 \otimes E_1 E_2 F_1) \otimes (p_2 - p_1 + s)w_s \in (1 \otimes E_1 E_2 F_2) \otimes F_{\Delta}w_s + Z.
$$
The proof follows from $E_1 E_2 F_1 = - E_1 F_1 E_2 = -(-2 - F_1 E_1) E_2=2 E_2 + F_1 E_1 E_2$, $E_1 E_2 F_2 = E_1 (-2 - F_2 E_2) = - 2 E_1 - E_1 F_2 E_2=- 2 E_1 + F_2 E_1 E_2$ and $(U(\mathfrak{g}) \otimes C(\mathfrak{p})\p^{+} \p^{+}) \otimes W \subset Z$.
\qed

\bigskip

From Lemma \ref{F}, Corollary \ref{E} and from the previous conclusions it follows that the vector space $\mathcal{A} \otimes_{\mathcal{B}} W$ is spanned by the set
\begin{align*}
\{ & (x \otimes F_2) \otimes w_s, (x \otimes E_1 F_2) \otimes w_s,  (x \otimes E_2 F_2) \otimes w_s,  (x \otimes E_1) \otimes w_s, \\
& (x \otimes F_1) \otimes w_{p_1 - p_2}, (x \otimes E_1 F_1) \otimes w_{p_1 - p_2}, \\
& (x \otimes E_2 F_1) \otimes w_{p_1 - p_2}, (x \otimes E_2) \otimes w_{p_1 - p_2} \, | \, s = 1, 2, \cdots, p_1 - p_2,  \\
& x \in \{ E_{1}^{n} F_{2}^{m}, (\text{ad}F)E_{1}^{n} F_{2}^{m}, \cdots, (\text{ad}F)^{n+m}E_{1}^{n} F_{2}^{m} \, | \, n, m \in \mathbb{N}_{0} \} \}.
\end{align*}

So far, we have reduced a big part of the Clifford algebra. Similarly, using the well known (and easy to prove) fact
$$
D (1 \otimes X) + (1 \otimes X) D = - 2 X \otimes 1, \quad \text{for} \quad X \in \p
$$
we can reduce a big part of the algebra $U(\g)$.

\begin{cor}\label{FiEi}
For $s \in \{ 1,2, \cdots, p_1 - p_2\}$ we have
\begin{align*}
(F_1 \otimes 1) \otimes (p_2 - p_1 + s) w_s & \in  (F_2 \otimes 1) \otimes F_{\Delta}w_s + Z \\
(E_2 \otimes 1) \otimes (p_2 - p_1 + s) w_s & \in  -(E_1 \otimes 1) \otimes F_{\Delta}w_s + Z.
\end{align*}
\end{cor}
\proof
Since $D(1 \otimes F_i) + (1 \otimes F_i)D = - 2 F_i \otimes 1$, for $i = 1, 2$, we have
\begin{align*}
(F_i \otimes 1) \otimes w & = - \frac{1}{2}(D(1 \otimes F_i) + (1 \otimes F_i)D) \otimes w \\
& = - \frac{1}{2}(D(1 \otimes F_i)) \otimes w - \frac{1}{2}\underbrace{((1 \otimes F_i) D \otimes w - (1 \otimes F_i) \otimes Dw)}_{\in Z} \\
&- \frac{1}{2}(1 \otimes F_i) \otimes \underbrace{Dw}_{= 0} \\
& \in - \frac{1}{2}D(1 \otimes F_i) \otimes w + Z.
\end{align*}
From lemma \ref{F} we get
$$
(p_2 - p_1 + s)(F_1 \otimes 1) \otimes w_s \in (F_2 \otimes 1) \otimes F_{\Delta}w_s + Z.
$$
Similarly, using $D(1 \otimes E_i) + (1 \otimes E_i)D = - 2E_i \otimes 1$ for $i = 1, 2$ and corollary \ref{E} we get
$$
(p_2 - p_1 + s)(E_2 \otimes 1) \otimes w_s \in -(E_1 \otimes 1) \otimes F_{\Delta}w_s + Z.
$$
\qed

One can easily show by induction that for each $n \in \mathbb{N}$ it holds
\begin{align}\label{calc}
E_2 F_2^{n} & = n F_{2}^{n-1} H_2 - n(n-1)F_{2}^{n-1} + F_{2}^{n} E_2 \\
F_2^{n}E_1 & = E_1 F_{2}^{n} - n F_{2}^{n-1}E \notag
\end{align}

Roughly speaking, the following proposition shows that we can replace the second highest weight vector of the $K$--type $V_{(n, -m)} \in U(\g)$ with the highest weight vectors of the $K$--types $V_{(n, -m)} \in U(\g)$ and $V_{(n-1, -(m-1))} \in U(\g)$.
\begin{prop}\label{l1}
For $s \in \{1, 2, \cdots,  p_1 - p_2 - 1\}$ and for nonnegative integers $m$ and $n$ we have
$$
((\text{ad}F)(E_{1}^{n} F_{2}^{m}) \otimes F_2) \otimes w_s \in \Span_{\mathbb{C}}\{ (E_{1}^{n} F_{2}^{m} \otimes F_2) \otimes w_{s+1}, (E_{1}^{n-1} F_{2}^{m-1} \otimes F_2) \otimes w_s \} + Z,
$$
where $E_1^{-1} = F_{2}^{-1} = 1$.
\end{prop}

\proof
From $[F, E_1] = E_2$, $[F, F_2] = - F_1$, $[E_1, E_2] = 0$ a $[F_1, F_2] = 0$ we get
$$
(\text{ad}F)(E_{1}^{n} F_{2}^{m}) = n \cdot E_{1}^{n-1} E_2 F_{2}^{m} - m E_{1}^{n} F_{2}^{m-1} F_1.
$$
From here and from \eqref{calc} it follows that
\begin{align}\label{(1)}
((\text{ad}F)(E_{1}^{n} F_{2}^{m}) \otimes F_2) \otimes w_s & = nm(E_{1}^{n-1} F_{2}^{m-1} H_2 \otimes F_2) \otimes w_s \\
& - nm(m-1)(E_{1}^{n-1} F_{2}^{m-1} \otimes F_2) \otimes w_s  \notag \\
&+ n(E_{1}^{n-1}F_{2}^{m} E_2 \otimes F_2) \otimes w_s \notag \\
&- m(E_{1}^{n} F_{2}^{m-1}F_1 \otimes F_2) \otimes w_s. \notag
\end{align}
From $(p_2 - p_1 + s)(F_1 \otimes 1) \otimes w_s \in (F_2 \otimes 1) \otimes w_{s+1} + Z$ (see corollary \eqref{FiEi} ) we get
\begin{align}\label{(2)}
& (p_2 - p_1 + s)(E_{1}^{n}F_{2}^{m-1}F_1 \otimes F_2) \otimes w_s \\
& \in (p_2 - p_1 + s)(E_{1}^{n}F_{2}^{m} \otimes F_2) \otimes w_{s+1} + Z. \notag
\end{align}
Furthermore, using $H_{2_{\Delta}} w_s = (h_{1_\Delta} + 2 h_{2_\Delta})w_s = (p_1 + 2p_2 + s - \frac{1}{2})w_s$,
$$\alpha(H_2) = \alpha(h_1 + 2 h_2) = - \frac{1}{2} E_1 F_1 - E_2 F_2 - \frac{3}{2},$$  and \eqref{pp} we get
\begin{align*}
(H_2 \otimes F_2) \otimes w_s & = (1 \otimes F_2)(H_{2_{\Delta}} - 1 \otimes \alpha(H_2)) \otimes w_s \\
& \in (p_1 + 2p_2 + s - 1) \cdot (1 \otimes F_2) \otimes w_s + Z \\
\end{align*}
Therefore
\begin{align}\label{(3)}
& (p_2 - p_1 + s)(E_{1}^{n-1} F_{2}^{m-1} H_2 \otimes F_2) \otimes w_s \\
& \in (p_2 - p_1 + s)(p_1 + 2p_2 + s - 1)(E_{1}^{n-1}F_{2}^{m-1} \otimes F_2) \otimes w_s + Z. \notag
\end{align}
Similarly, using $(p_2 - p_1 + s)(E_2 \otimes 1) \otimes w_s \in - (E_1 \otimes 1) \otimes w_{s+1} + Z$, and \eqref{calc}
we get
\begin{align*}
& (p_2 - p_1 + s)(E_{1}^{n-1} F_{2}^{m} E_2 \otimes F_2) \otimes w_s \in - (E_{1}^{n-1} F_{2}^{m} E_1 \otimes F_2) \otimes w_{s+1} + Z \\
& = - (E_{1}^{n-1}(E_1 F_{2}^{m} - m F_{2}^{m-1} E) \otimes F_2) \otimes w_{s+1} + Z.
\end{align*}
Since
\begin{align*}
(E \otimes F_2) \otimes w_{s+1} & = (1 \otimes F_2)(E_{\Delta} - 1 \otimes \alpha(E)) \otimes w_{s+1} \\
& = (1 \otimes F_2)\underbrace{(E_{\Delta} \otimes w_{s+1} - (1 \otimes 1) \otimes E_{\Delta} w_{s+1})}_{\in Z} \\
& + (1 \otimes F_2) \otimes E_{\Delta} w_{s+1} - 1 \otimes F_2 \cdot (- \frac{1}{2} E_1 F_2) \otimes w_{s+1} \\
& \in s(p_1 - p_2 - s)(1 \otimes F_2) \otimes w_s + Z,
\end{align*}
it follows that
\begin{align}\label{(4)}
& (p_2 - p_1 + s)(E_{1}^{n-1} F_{2}^{m} E_2 \otimes F_2) \otimes w_s  \in - (E_{1}^{n} F_{2}^{m} \otimes F_2) \otimes w_{s+1} \\
& + ms(p_1 - p_2 - s)(E_{1}^{n-1} F_{2}^{m-1} \otimes F_2) \otimes w_s + Z. \notag
\end{align}
The proof follows from \eqref{(1)}, \eqref{(2)}, \eqref{(3)} and \eqref{(4)}.
\qed

The next proposition shows that a similar result holds for $s = p_1 - p_2$:

\begin{prop}\label{l2}
For nonnegative integers $m$ and $n$ we have
\begin{align*}
& ((\text{ad}F)(E_{1}^{n} F_{2}^{m}) \otimes F_2) \otimes w_{p_1 - p_2} \\
& \in \Span_{\mathbb{C}}\{ (E_{1}^{n} F_{2}^{m} \otimes F_1) \otimes w_{p_1 - p_2}, (E_{1}^{n-1} F_{2}^{m-1} \otimes F_2) \otimes w_{p_1 - p_2} \} + Z,
\end{align*}
where $E_{1}^{-1} = F_{2}^{-1} = 1$.
\end{prop}

\begin{proof}
From
$$
(F_1 \otimes F_2) \otimes w_{p_1 - p_2} = \underbrace{(F_1 \otimes F_2 - F_2 \otimes F_1) \otimes w_{p_1 - p_2}}_{\in Z} + (F_2 \otimes F_1) \otimes w_{p_1 - p_2}
$$
we get
\begin{equation}\label{(5)}
(E_{1}^{n} F_{2}^{m-1} F_1 \otimes F_2) \otimes w_{p_1 - p_2} \in (E_{1}^{n}F_{2}^{m} \otimes F_1) \otimes w_{p_1 - p_2} + Z
\end{equation}
Furthermore, from the proof of the proposition \ref{l1} we get
\begin{equation}\label{(6)}
(E_{1}^{n-1} F_{2}^{m-1} H_2 \otimes F_2) \otimes w_{p_1 - p_2} \in (2 p_1 + p_2 - 1) \cdot (E_{1}^{n-1} F_{2}^{m-1} \otimes F_2) \otimes w_{p_1 - p_2} + Z.
\end{equation}
Since
$$
(E_2 \otimes F_2) \otimes w_{p_1 - p_2} = \underbrace{(E_2 \otimes F_2 + E_1 \otimes F_1) \otimes w_{p_1 - p_2}}_{\in Z} - (E_1 \otimes F_1) \otimes w_{p_1 - p_2},
$$
then
\begin{align*}
(E_{1}^{n-1} F_{2}^{m} E_2 \otimes F_2) \otimes w_{p_1 - p_2}  \in & -(E_{1}^{n-1} F_{2}^{m} E_1 \otimes F_1) \otimes w_{p_1 - p_2} + Z \\
 = & - (E_{1}^{n-1}(E_1 F_{2}^{m} - m F_{2}^{m-1} E) \otimes F_1) \otimes w_{p_1 - p_2} + Z \\
 = &- (E_{1}^{n} F_{2}^{m} \otimes F_1) \otimes w_{p_1 - p_2} \\
 & + m(E_{1}^{n-1} F_{2}^{m-1} E \otimes F_1) \otimes w_{p_1 - p_2} + Z.
\end{align*}
Furthermore, using \eqref{pp}, and $E_{\Delta} w_{p_1 - p_2} = (p_1 - p_2 - 1)w_{p_1 - p_2 - 1}$ we get
\begin{align*}
(E \otimes F_1) \otimes w_{p_1 - p_2} & = (1 \otimes F_1)(E_{\Delta} - 1 \otimes \alpha(E)) \otimes w_{p_1 - p_2} \\
& \in (p_1 - p_2 - 1)(1 \otimes F_1) \otimes w_{p_1 - p_2 - 1} - (1 \otimes F_2) \otimes w_{p_1 - p_2} + Z.
\end{align*}
Now, from
$$
(p_2 - p_1 + (p_1 - p_2 - 1))(1 \otimes F_1) \otimes w_{p_1 - p_2 - 1} \in  (1 \otimes F_2) \otimes w_{p_1 - p_2} + Z
$$
we get
\begin{align*}
(E \otimes F_1) \otimes w_{p_1 - p_2} & \in (p_2 - p_1 + 1)(1 \otimes F_2) \otimes w_{p_1 - p_2} - (1 \otimes F_2) \otimes w_{p_1 - p_2} + Z \\
& = (p_2 - p_1)(1 \otimes F_2) \otimes w_{p_1 - p_2} + Z
\end{align*}
and therefore
\begin{align}\label{(7)}
& (E_{1}^{n-1} F_{2}^{m} E_2 \otimes F_2) \otimes w_{p_1 - p_2} \\
& \in -(E_{1}^{n} F_{2}^{m} \otimes F_1) \otimes w_{p_1 - p_2} + m(p_2 - p_1)(E_{1}^{n-1} F_{2}^{m-1} \otimes F_2) \otimes w_{p_1 - p_2} + Z. \notag
\end{align}
The proof follows from \eqref{(1)}, \eqref{(5)}, \eqref{(6)} and \eqref{(7)}.
\end{proof}
Finally, from propositions \ref{l1} and \ref{l2} and previous conclusions we get
\begin{theorem}\label{si}
One generating system of the vector space $\mathcal{A} \otimes_{\mathcal{B}} W$ is given by
\begin{align*}
\{ & (E_{1}^{n} F_{2}^{m} \otimes F_2) \otimes w_s, (E_{1}^{n} F_{2}^{m} \otimes E_1 F_2) \otimes w_s, (E_{1}^{n} F_{2}^{m} \otimes E_2 F_2) \otimes w_s, \\
& (E_{1}^{n} F_{2}^{m} \otimes E_1) \otimes w_s, \\
& (x \otimes F_1) \otimes w_{p_1 - p_2}, (x \otimes E_1 F_1) \otimes w_{p_1 - p_2},  (x \otimes E_2 F_1) \otimes w_{p_1 - p_2}, (x \otimes E_2) \otimes w_{p_1 - p_2} \\
& s = 1, 2, \cdots, p_1 - p_2, \\
& x \in \{ E_{1}^{n} F_{2}^{m}, (\text{ad}F)E_{1}^{n} F_{2}^{m}, \cdots, (\text{ad}F)^{n+m}E_{1}^{n} F_{2}^{m} \, | \, n, m \in \mathbb{N}_{0} \} \}.
\end{align*}
\end{theorem}
There are no more obvious relations between the elements in \ref{si}. We will show that they form a basis  of $\mathcal{A} \otimes_{\mathcal{B}}W$ a little later.

\section{The construction of the nonholomorphic discrete series representations of the group $SU(2,1)$ via Dirac induction}

From Theorem \ref{basis1} it is easy to write down one basis for $X \otimes S$. However, we will give another basis for the space $X \otimes S$ which will be more useful to prove what we claim. First we will give bases for some finite-dimensional subspaces of the space $X \otimes 1$.
Let us denote
$$
V_{(p_1 + n, p_2 - m)} = \text{span}_{\mathbb{C}} \{ v_{p_1 + n, p_2 - m}^{(p_1 + n) - (p_2 - m) - 2s} \, | \, s \in \{ 0,1, \cdots, (p_1 + n) - (p_2 - m) \}\}.
$$
From \eqref{paction} it follows that one basis for the vector space
$$
V_{k,l} = \displaystyle \sum_{0 \leq n \leq k, 0 \leq m \leq l} V_{(p_1 + n, p_2 - m)} \otimes 1
$$
is given by
\begin{equation}\label{finite_basis}
\{ (F^s E_{1}^{n} F_{2}^{m} \otimes F_2) \cdot w_1 \, | \, 0 \leq n \leq k, \quad 0 \leq m \leq l, \quad 0 \leq s \leq (p_1 + n) - (p_2 - m)\}.
\end{equation}
Our goal is to get a basis for $X \otimes S$ which is similar to the generating system from the theorem \ref{si}.
The next result is analogous to Proposition \ref{F}.
\begin{lemma}\label{1oF_1}
For $s \in \{ 1, 2, \cdots, p_1 - p_2\}$ we have $(p_2 - p_1 + s)(1 \otimes F_1)\cdot w_s = (1 \otimes F_2) \cdot F_{\Delta}w_s$ and $(p_2 - p_1 + s)(1 \otimes E_2) \cdot w_s = - (1 \otimes E_1) \cdot F_{\Delta} w_s$.
\end{lemma}
\begin{proof}
For $s \in \{ 1, 2, \cdots, p_1 - p_2 - 1 \}$ we have
\begin{align*}
&(1 \otimes F_2) \cdot w_{s+1} = (1 \otimes F_2)(v_{p_1, p_2}^{p_1 - p_2 - 2(s+1)} \otimes E_1 - (p_1 - p_2 - s)v_{p_1, p_2}^{p_1 - p_2 - 2s} \otimes E_2) \\
& = 2(p_1 - p_2 - s)v_{p_1, p_2}^{p_1 - p_2 - 2s} \otimes 1, \\
&(1 \otimes F_1) \cdot w_s  = (1 \otimes F_1)(v_{p_1, p_2}^{p_1 - p_2 - 2s} \otimes E_1 - (p_1 - p_2 - (s-1))v_{p_1, p_2}^{p_1 - p_2 - 2(s-1)} \otimes E_2) \\
& = -2 v_{p_1, p_2}^{p_1 - p_2 - 2s} \otimes 1, \\
&(1 \otimes E_1) \cdot w_{s+1}  = (1 \otimes E_1)(v_{p_1, p_2}^{p_1 - p_2 - 2(s+1)} \otimes E_1 - (p_1 - p_2 - s)v_{p_1, p_2}^{p_1 - p_2 - 2s} \otimes E_2) \\
& = -(p_1 - p_2 - s)v_{p_1, p_2}^{p_1 - p_2 - 2s} \otimes E_1 \wedge E_2, \\
&(1 \otimes E_2) \cdot w_s = (1 \otimes E_2)(v_{p_1, p_2}^{p_1 - p_2 - 2s} \otimes E_1 - (p_1 - p_2 - (s-1))v_{p_1, p_2}^{p_1 - p_2 - 2(s-1)} \otimes E_2) \\
& = - v_{p_1, p_2}^{p_1 - p_2 - 2s} \otimes E_1 \wedge E_2.
\end{align*}
This implies the claim for $s \in \{ 1, 2, \cdots, p_1 - p_2 - 1\}$. The case $s = p_1 - p_2$ is obvious.
\end{proof}
Furthermore, we have the following analogue of Corollary \ref{FiEi}:
\begin{lemma}\label{F_1o1}
For $s \in \{ 1,2, \cdots, p_1 - p_2 \}$ we have
$$
(p_2 - p_1 + s)(F_1 \otimes 1) \cdot w_s = (F_2 \otimes 1) \cdot F_{\Delta} w_s
$$
and
$$
(p_2 - p_1 + s)(E_2 \otimes 1) \cdot w_s \in - (E_1 \otimes 1) \cdot F_{\Delta} w_s.
$$
\end{lemma}
\begin{proof}
Since $D(1 \otimes F_i) + (1 \otimes F_i)D = - 2 F_i \otimes 1$, for $i = 1, 2 $, then for all $w \in W$ we get
$$
(F_i \otimes 1) \cdot w  = -\frac{1}{2}(D(1 \otimes F_i) + (1 \otimes F_i)D) \cdot w = - \frac{1}{2}D(1 \otimes F_i) \cdot w.
$$
From this and from Lemma \ref{1oF_1} it follows that
$$
(p_2 - p_1 + s)(F_1 \otimes 1) \cdot w_s = (F_2 \otimes 1) \cdot F_{\Delta}w_s.
$$
Similarly, using $D(1 \otimes E_2) + (1 \otimes E_2)D = - 2 E_2 \otimes 1$ and Lemma \ref{1oF_1} we get
$$
(p_2 - p_1 + s)(E_2 \otimes 1) \cdot w_s = - (E_1 \otimes 1) \cdot F_{\Delta}w_s.
$$
\end{proof}

Similarly as in the previous section, we would like to replace the second highest weight vector of the $K$--types in $U(\g)$ with the highest weight vectors of other $K$--types in $U(\g)$.

\begin{prop}\label{FF2}
For $n,m \in \mathbb{N}_{0}$ and $s \in \{1, 2, \cdots, p_1 - p_2 - 1\}$ we have
$$
(F E_{1}^{n} F_{2}^{m} \otimes F_2) \cdot w_s \in \Span_{\mathbb{C}}\{ (E_{1}^{n-1}F_{2}^{m-1} \otimes F_2) \cdot w_s, (E_{1}^{n} F_{2}^{m} \otimes F_2) \cdot w_{s+1}\},
$$
where $E_{1}^{-1} = F_{2}^{-1} = 1$.
\end{prop}

\begin{proof}
From $\text{ad}(F)(E_{1}^{n} F_{2}^{m}) = F E_{1}^{n} F_{2}^{m} - E_{1}^{n} F_{2}^{m}F$, in a similar way to how it was done in \eqref{(1)}, we get

\begin{align}\label{1}
 (F E_{1}^{n} F_{2}^{m} \otimes F_2) \cdot w_s & = nm (E_{1}^{n-1} F_{2}^{m-1} H_2 \otimes F_2) \cdot w_s \\
& - mn(m-1)(E_{1}^{n-1} F_{2}^{m-1} \otimes F_2) \cdot w_s \notag \\
& + n (E_{1}^{n-1} F_{2}^{m} E_2 \otimes F_2) \cdot w_s  - m (E_{1}^{n} F_{2}^{m-1} F_1 \otimes F_2) \cdot w_s \notag \\
& + (E_{1}^{n} F_{2}^{m} F \otimes F_2) \cdot w_s. \notag
\end{align}

Using lemma \ref{F_1o1} and \eqref{calc}, we get for $s \in \{ 1, 2, \cdots, p_1 - p_2 -1 \}$

\begin{align}\label{2}
(p_2 - p_1 + s) (E_{1}^{n} F_{2}^{m-1} F_1 \otimes F_2) \cdot w_s  & = (E_{1}^{n} F_{2}^{m} \otimes F_2) \cdot w_{s+1}  \\
(p_2 - p_1 + s)(E_{1}^{n-1} F_{2}^{m} E_2 \otimes F_2) \cdot w_s  & = - (E_{1}^{n-1} F_{2}^{m} E_1 \otimes F_2) \cdot w_{s+1} \notag\\
& = - (E_{1}^{n-1}(E_{1} F_{2}^{m} - m F_{2}^{m-1} E) \otimes F_2) \cdot w_{s+1} \notag \\
 & = - (E_{1}^{n} F_{2}^{m} \otimes F_2) \cdot w_{s+1} \notag \\
& + m (E_{1}^{n-1} F_{2}^{m-1} E \otimes F_2) \cdot w_{s+1}. \notag
\end{align}
Furthermore, for $s \in \{ 1,2, \cdots, p_1 - p_2 - 1\}$ the following identity holds
\begin{align*}
E \otimes F_{2} \cdot w_{s+1} & = E \otimes F_2 \cdot (v_{p_1, p_2}^{p_1 - p_2 - 2(s+1)} \otimes E_1 - (p_1 - p_2 - s)v_{p_1, p_2}^{p_1 - p_2 - 2s} \otimes E_2) \\
& = 2 (p_1 - p_2 - s) E v_{p_1, p_2}^{p_1 - p_2 - 2s} \otimes 1 \\
& = 2 (p_1 - p_2 - s)s(p_1 - p_2 - s + 1) v_{p_1, p_2}^{p_1 - p_2 - 2(s-1)} \otimes 1 \\
& = (p_1 - p_2 - s)s \cdot (1 \otimes F_2) \cdot w_s.
\end{align*}
From this we get
\begin{align}\label{3}
(p_2 - p_1 + s)(E_{1}^{n-1} F_{2}^{m} E_2 \otimes F_2) \cdot w_s & = - (E_{1}^{n} F_{2}^{m} \otimes F_2) \cdot w_{s+1} \\
& + ms(p_1 - p_2 - s) (E_{1}^{n-1} F_{2}^{m-1} \otimes F_2) \cdot w_s. \notag
\end{align}
Now, from $H_2 = h_1 + 2 h_2$ and \eqref{kaction} it follows that
\begin{align}\label{4}
(H_2  \otimes F_2) \cdot w_s & = 2(p_1 - p_2 - (s - 1)) H_2 v_{p_1, p_2}^{p_1 - p_2 - 2(s-1)} \otimes 1  \\
& = 2(p_1 - p_2 - (s - 1))(p_1 + 2 p_2 + s-1) v_{p_1, p_2}^{p_1 - p_2 - 2(s-1)} \otimes 1 \notag \\
& = (p_1 + 2p_2 + s-1) (1 \otimes F_2) \cdot w_s. \notag
\end{align}
Furthermore, we have
\begin{align*}
(F \otimes F_2) \cdot w_s  & = 2 (p_1 - p_2 - (s-1)) F v_{p_1, p_2}^{p_1 - p_2 - 2(s-1)} \otimes 1 \\
& = 2 (p_1 - p_2 - (s-1)) v_{p_1, p_2}^{p_1 - p_2 - 2s} \otimes 1 \\
& = \frac{p_1 - p_2 - (s-1)}{p_1 - p_2 - s} (1 \otimes F_2) \cdot w_{s+1},
\end{align*}
from where we get
\begin{equation}\label{5}
(p_1 - p_2 - s) (F \otimes F_2) \cdot w_s = (p_1 - p_2 - (s-1)) (1 \otimes F_2) \cdot w_{s+1}.
\end{equation}
Finally, from \eqref{1}, \eqref{2}, \eqref{3}, \eqref{4} and \eqref{5} we get
\begin{align*}
& (p_2 - p_1 + s) (F E_{1}^{n} F_{2}^{m} \otimes F_2) \cdot w_s \\
 = & (p_2 - p_1 + s)mn(p_1 + 2p_2 - m) (E_{1}^{n-1} F_{2}^{m-1} \otimes F_2) \cdot w_s \\
& -(n + m + p_1 - p_2 - (s-1))(E_{1}^{n} F_{2}^{m} \otimes F_2) \cdot w_{s+1}.
\end{align*}
Therefore, for $s \in \{ 1,2, \cdots, p_1 - p_2-1 \}$ we have $$
(F E_{1}^{n} F_{2}^{m} \otimes F_2) \cdot w_s \in \Span_{\mathbb{C}}\{ (E_{1}^{n-1}F_{2}^{m-1} \otimes F_2) \cdot w_s, (E_{1}^{n} F_{2}^{m} \otimes F_2) \cdot w_{s+1}\}.
$$
\end{proof}
A similar result holds for $s=p_1 - p_2$.
\begin{lemma}\label{FF2p}
Let $n$ and $m$ be nonnegative integers. Then
$$
(F E_{1}^{n} F_{2}^{m} \otimes F_2) \cdot w_{p_1 - p_2} \in \Span_{\mathbb{C}}\{ (E_{1}^{n-1}F_{2}^{m-1} \otimes F_2) \cdot w_{p_1 - p_2}, (E_{1}^{n} F_{2}^{m} \otimes F_1) \cdot w_{p_1 - p_2}\}.
$$
where $E_{1}^{-1} = F_{2}^{-1} = 1$.
\end{lemma}

\begin{proof}
From the facts $(E_1 \otimes F_1 + E_2 \otimes F_2) \cdot w_{p_1 - p_2} = 0$, $(F_1 \otimes F_2 - F_2 \otimes F_1) \cdot w_{p_1 - p_2} = 0$, $(F \otimes F_2) \cdot w_{p_1 - p_2} = 2 v_{p_1, p_2}^{-p_1 + p_2} \otimes 1 = - (1 \otimes F_1) \cdot w_{p_1 - p_2}$, from \eqref{1}, \eqref{4}, \eqref{calc} and $(E \otimes F_1) \cdot w_{p_1 - p_2} = (p_2 - p_1)(1 \otimes F_2) \cdot w_{p_1 - p_2}$ 
we get
\begin{align*}
(F E_{1}^{n} F_{2}^{m} \otimes F_2) \cdot w_{p_1 - p_2}
= & mn(2p_2 + p_1 - m) (E_{1}^{n-1} F_{2}^{m-1} \otimes F_2) \cdot w_{p_1 - p_2} \\
 & - (m+n + 1) (E_{1}^{n} F_{2}^{m} \otimes F_1) \cdot w_{p_1 - p_2}.
\end{align*}
This finishes the proof.
\end{proof}
\begin{lemma}\label{lem}
For $n, m, s \in \mathbb{N}_{0}$ we have
$$
(\text{ad}F)^{s}(E_{1}^{n} F_{2}^{m}) = F^s E_{1}^{n} F_{2}^{m} - \sum_{i = 1}^{s} \binom{s}{i}(\text{ad}F)^{s-i}(E_{1}^{n} F_{2}^{m}) F^{i}.
$$
\end{lemma}

\begin{proof}
The proof follows easily by induction on $s$, for fixed $n,m \in \mathbb{N}_{0}$.
\end{proof}

\begin{lemma}\label{FF1p}
For nonnegative integers $n,m$ and $s$ the following identity holds
$$
(F^s E_{1}^{n} F_{2}^{m} \otimes F_1) \cdot w_{p_1 - p_2} = ((\text{ad}F)^{s}(E_{1}^{n} F_{2}^{m}) \otimes F_1) \cdot w_{p_1 - p_2}.
$$
\end{lemma}

\begin{proof}
From Lemma \ref{lem} it follows that
\begin{align*}
(F^s E_{1}^{n} F_{2}^{m} \otimes F_1) \cdot w_{p_1 - p_2} = & \left( \sum_{i=0}^{s} \binom{s}{i} ((\text{ad}F)^{s-i}(E_{1}^{n} F_{2}^{m})) F^i \otimes F_1 \right) \cdot w_{p_1 - p_2} \\
& = ((\text{ad}F)^{s}(E_{1}^{n} F_{2}^{m}) \otimes F_1) \cdot w_{p_1 - p_2},
\end{align*}
since $F^i \otimes F_1 \cdot w_{p_1 - p_2} = 0$ for $i \geq 1$.
\end{proof}

Now, from Lemmas \ref{FF2}, \ref{FF2p}, \ref{FF1p} and from \eqref{finite_basis} we see that the set
\begin{align*}
B = \{ (E_{1}^{n} F_{2}^{m} \otimes F_2) \cdot w_s, (x \otimes F_1) \cdot w_{p_1 - p_2} \, | \, & s \in \{ 1,2, \cdots, p_1 - p_2 \}, \\ & x \in \{ E_{1}^{n} F_{2}^{m}, (\text{ad}F)E_{1}^{n} F_{2}^{m}, \cdots, \\
& (\text{ad}F)^{n+m}E_{1}^{n} F_{2}^{m} \}, \\ & 0 \leq n \leq k, 0 \leq m \leq l \} \}
\end{align*}
spans the vector space $V_{k,l}$.
Since $B$ has exactly
$$
\sum_{0 \leq n \leq k, 0 \leq m \leq l} (p_1 - p_2) + n + m + 1
$$
elements, which is the dimension of the vector space $V_{k,l}$, $B$ is the basis of $V_{k,l}$. \newline
Since $X \otimes 1 = \sum_{n,m \in \mathbb{N}_{0}} V_{(p_1 + n, p_2 - m)} \otimes 1$, one basis for $X \otimes 1$ is given by
\begin{align*}
& \{ (E_{1}^{n} F_{2}^{m} \otimes F_2) \cdot w_s, (x \otimes F_1) \cdot w_{p_1 - p_2} \, | \, s \in \{ 1,2, \cdots, p_1 - p_2 \}, \\ & x \in \{ E_{1}^{n} F_{2}^{m}, (\text{ad}F)E_{1}^{n} F_{2}^{m}, \cdots, (\text{ad}F)^{n+m}E_{1}^{n} F_{2}^{m} \} \, | \, n, m \in \mathbb{N}_{0} \} \},
\end{align*}
From here, using $(1 \otimes E_1 E_2) \cdot w_s = 0$, we get that one basis for the vector space $X \otimes S$ is given by
\begin{align*}
\{ & (E_{1}^{n} F_{2}^{m} \otimes F_2) \cdot w_s, (E_{1}^{n} F_{2}^{m} \otimes E_1 F_2) \cdot w_s, (E_{1}^{n} F_{2}^{m} \otimes E_2 F_2) \cdot w_s,  (E_{1}^{n} F_{2}^{m} \otimes E_1) \cdot w_s, \\
& (x \otimes F_1) \cdot w_{p_1 - p_2}, (x \otimes E_1 F_1) \cdot w_{p_1 - p_2},  (x \otimes E_2 F_1) \cdot w_{p_1 - p_2}, (x \otimes E_2) \cdot w_{p_1 - p_2} \\
& s = 1, 2, \cdots, p_1 - p_2, \\
& x \in \{ E_{1}^{n} F_{2}^{m}, (\text{ad}F)E_{1}^{n} F_{2}^{m}, \cdots, (\text{ad}F)^{n+m}E_{1}^{n} F_{2}^{m} \, | \, n, m \in \mathbb{N}_{0} \} \}.
\end{align*}

The action map $\phi : \mathcal{A} \otimes W \longrightarrow X \otimes S$
$$
\phi(a \otimes w) = a \cdot w, \quad a \in \mathcal{A}, w \in W
$$
maps the generating system of $\mathcal{A} \otimes_{\mathcal{B}} W$ from Theorem \ref{si} into the basis of the space $X \otimes S$, which implies that generating system is a basis of the vector space $\mathcal{A} \otimes_{\mathcal{B}} W$. Finally, we get:
\begin{theorem}
Let $X$ be the vector space of the nonholomorphic discrete series representation of the group $SU(2,1)$ and let $W$ be the Dirac cohomology of $X$. Then the
$(\mathcal{A}, \tilde{K})$--module $\text{Ind}_{D}(W) = \mathcal{A} \otimes_{\mathcal{B}} W$ is isomorphic to the $(\mathcal{A}, \tilde{K})$--module $X \otimes S$.
\end{theorem}

\begin{remark}
The algebra $U(\ka_{\Delta})(C(\p)^{K} + \mathcal{I})$ is a proper subalgebra of the algebra $U(\ka_{\Delta}) \mathcal{A}^{K}$, but it is still not clear if nonholomorphic discrete series representations of the group $SU(2,1)$ can be constructed via intermediate Dirac induction. 
\end{remark}

\end{document}